\documentclass[a4paper,11pt]{amsart}

\usepackage{amsfonts,amsmath,mathrsfs,amssymb,amsthm,amscd,latexsym,amstext,amsxtra}

\usepackage[usenames]{color}
\usepackage[all]{xy}
\usepackage{enumerate}
\xyoption{all}
\usepackage{srcltx} %per fer cerca inversa
\usepackage{amsthm, amsmath, amssymb,latexsym}
\usepackage{amsfonts,epsfig,amscd}
\usepackage[]{fontenc}
\usepackage{xy}
\usepackage{enumerate}
\usepackage[]{fontenc}
\usepackage{xy}
\newtheorem{theorem}{Theorem}[section]
\newtheorem{lemma}[theorem]{Lemma}

\newtheorem{proposition}[theorem]{Proposition}
\newtheorem{remark}[theorem]{Remark}
\newtheorem{definition}[theorem]{Definition}
\newtheorem{example}[theorem]{Example}

\newcommand{\nc}{\newcommand}
\nc{\cH}{{\mathcal H}}
\nc{\cA}{{\mathcal A}}
\nc{\cG}{{\mathcal G}}
\nc{\cC}{{\mathcal C}}
\nc{\cD}{{\mathcal D}}
\nc{\cO}{{\mathcal O}}
\nc{\cI}{{\mathcal I}}
\nc{\cB}{{\mathcal B}}
\nc{\cY}{{\mathcal Y}}
\nc{\cK}{{\mathcal K}}
\nc{\cX}{{\mathcal X}}
\nc{\cS}{{\mathcal S}}
\nc{\cE}{{\mathcal E}}
\nc{\cF}{{\mathcal F}}
\nc{\cZ}{{\mathcal Z}}
\nc{\cQ}{{\mathcal Q}}
\nc{\cN}{{\mathcal N}}
\nc{\cP}{{\mathcal P}}
\nc{\cL}{{\mathcal L}}
\nc{\cM}{{\mathcal M}}
\nc{\cT}{{\mathcal T}}
\nc{\cW}{{\mathcal W}}
\nc{\cU}{{\mathcal U}}
\nc{\cJ}{{\mathcal J}}
\nc{\cV}{{\mathcal V}}
\nc{\bH}{{\mathbb H}}
\nc{\bA}{{\mathbb A}}
\nc{\bG}{{\mathbb G}}
\nc{\bC}{{\mathbb C}}
\nc{\bO}{{\mathbb O}}
\nc{\bI}{{\mathbb I}}
\nc{\bB}{{\mathbb B}}
\nc{\bY}{{\mathbb Y}}
\nc{\bK}{{\mathbb K}}
\nc{\bX}{{\mathbb X}}
\nc{\bS}{{\mathbb S}}
\nc{\bE}{{\mathbb E}}
\nc{\bF}{{\mathbb F}}
\nc{\bZ}{{\mathbb Z}}
\nc{\bQ}{{\mathbb Q}}
\nc{\bN}{{\mathbb N}}
\nc{\bP}{{\mathbb P}}
\nc{\bL}{{\mathbb L}}
\nc{\bM}{{\mathbb M}}
\nc{\bT}{{\mathbb T}}
\nc{\bW}{{\mathbb W}}
\nc{\bU}{{\mathbb U}}
\nc{\bD}{{\mathbb D}}
\nc{\bJ}{{\mathbb J}}
\nc{\bV}{{\mathbb V}}
\nc{\bbZ}{{\mathbb Z}}
\nc{\bR}{{\mathbb R}}
\nc{\fr}{{\rightarrow}}
\nc{\co}{{\nabla}}

\nc{\cu}{{\overlineline{\nabla}}}
\title {On rational maps from the product of two general curves}
\author{Yongnam Lee and Gian Pietro Pirola}
\date{}

\address{Department of Mathematical Sciences, KAIST, 291 Daehak-ro, Yuseong-gu, Daejon 305-701, Korea, and
Korea Institute for Advanced Study, Seoul 130-722, Korea}

\email{ynlee@kaist.ac.kr}

\address{Dipartimento di Matematica, Universit\`a di Pavia via Ferrata 1, 27100 Pavia, Italia}

\email{gianpietro.pirola@unipv.it}


\begin{document}

\begin{abstract}
This paper treats dominant rational maps from the product of two very general curves to nonsingular projective surfaces. Combining the result in \cite{BaPi} we prove that the product of two very general curves of genus $g\geq 7$ and $g'\geq 3$ does not admit dominant rational maps of degree $> 1$ if the image surface is non-ruled. We also treat the case of the 2-symmetric product of a curve.
\end{abstract}
\maketitle

%%%%%%%% pag  2 inizio %%%%%%%%%%%%%%%%%%

%%%%%%%%%%%%%%%%%%%%%%%%%%%
%%%%%%%%%%%%%%%%%%%%%%%%%%%%%%%%%%%%%%%%%%%%%%%%

\section {Introduction}

Let $X\subset \bC\bP^3$ be a very general smooth surface of degree $\geq 5,$ and $\bC(X)$ be its function field. In \cite{LeeP} we proved that every proper subfield
 $\bC\subset K\subset \bC(X)$ is a pure transcendental extension of $\bC.$  This means that if $X$ is very general and  $F:X\dashrightarrow S$ is a dominant map which is not birational, then either $S$ is a point, a projective line or a rational surface.
We recall that a very general element of a variety $U$ has the property $P$ if $P$ holds in the complement of a union of countably many proper subvarieties of $U$.
Let $X$ be smooth complex projective variety of general type. The dominant rational maps of finite degree $X\dashrightarrow Y$ to smooth varieties of general type, up to birational equivalence of $Y$ form a finite set $M(X)$. The proof follows from the approach of Maehara \cite{Ma}, combined with the results of Hacon and McKernan \cite{HM}, of Takayama \cite {Ta}, and of Tsuji \cite{Ts}.

Motivated by this finiteness theorem for dominant rational maps on a variety of general type, and by the results obtained in \cite{BaPi} and \cite{LeeP}, in this paper we study the case of the product of two very general smooth curves $X=C\times D$ of genus $g_C$ and $g_D$ respectively.  The $2$-symmetric
product of $C$, $X=C_2$, is also treated. The product case has previously been studied in \cite{BaPi}.  It was proved there that if
$g_C\geq 7$, $g_D\geq 2$, and  $S\neq C\times D$ is of general type then a dominant rational map $F:C\times D\dashrightarrow S$ does not exist. Here we complete the analysis by considering  surfaces $S$ of Kodaira dimension ${\rm kod}(S)=0$ and $1.$  Our main result is
\begin{theorem}
($=$Theorem 3.1) Let $C$ and $D$ be very general curves of genus $g_C\geq 7$ and $g_D\geq 3$ respectively. Let
$F: C\times D\dashrightarrow S$  be a dominant rational map of degree $>1$  where $S$ is a smooth projective surface.
Then $S$ is a ruled surface, that is   ${\rm kod}(S)=-\infty.$
\end{theorem}

We recall Riemann-Hurwitz Theorem (Chapter XXI in \cite{ACG}), which says that if $\phi: C\to \bar C$ is a non-constant morphism from a very general curve $C$ of genus $g>1$ onto a curve $\bar C$, then either $\phi$ is birational, or else $\bar C$ is rational. As in  \cite{gp,LeeP}  Hodge theory and deformation theory are the two main methods used to handle our problem. The analysis of the Hodge structure  of \cite{BaPi}  allows to reduce the problem to the case where the geometric genus and the irregularity of $S$ are zero:
$p_g(S)=q(S)=0.$ Therefore we have to deal with the case where $S$ is an elliptic surface.
The main new technical obstacle comes out from the fact that the fundamental group $C\times D$ is not abelian.
When $S$ has Kodaira dimension $1,$
we cannot infer directly neither that the elliptic surface $S$ has bounded topology nor that its moduli space has bounded dimension.
To overcome this problem we begin by proving that the first homology group of $S$, $H_1(S,\bZ),$ vanishes.
Then we use the result of the deformation of curves on elliptic surfaces with multiple fibers proved in \cite{LeeP} to obtain a contradiction.
The main ingredient we use to show that $H_1(S,\bZ)=0$ is Theorem~\ref{kodsh} which provides some restriction on the Hodge structure of certain abelian covering of $C\times D.$
Theorem~\ref{kodsh} seems to be of independent interest and to deserve further development.

The last section deals with
the case of the $2$-symmetric product of a curve. This case is simpler, by using a slightly improvement on the deformation of curves on elliptic surfaces $S$ we can prove (we do not attempt to find the optimal genus)
\begin{theorem} \label{1,1}
($=$Theorem 4.5) Let $C$ be a very general curve of genus $\geq 10$. If $f: C_2 \dashrightarrow S$ is a dominant rational map of degree $> 1$ then
$S$ is a rational surface.
\end{theorem}

In this paper we work on the field of complex numbers. We use the customary notation in algebraic geometry. For a smooth projective surface $S$, we let $p_g(S)=h^{2,0}(S)$ be the geometric genus, $q(S)=h^{1.0}(S)$ be the irregularity of $S$, and $alb(S)$ be the Albanese variety of $S$. We use the notation $\equiv$ for a linear equivalence of divisors.

\subsection*{Acknowledgements}
This work was initiated when the first named author visited University of Pavia supported by INdAM (GNSAGA) and ended when the second named author visited KAIST. They would like to thank University Pavia and KAIST for the hospitality. We thank the referee for several useful suggestions and remarks. The first named author is partially supported by the National Research Foundation of Korea(NRF) funded by the Korean government(MSIP)(No.2013006431) and (No.2013042157).
The second named author is partially supported by INdAM (GNSAGA); PRIN 2012 \emph{``Moduli, strutture geometriche e loro applicazioni''} and FAR 2013 (PV) \emph{``Variet\`a algebriche, calcolo algebrico, grafi orientati e topologici''}

\medskip
 %%%%%%%%%%%%%%%
%%%%%%%%%%%%%%%%%%%%%%%%%%%%%%%%%%%%%%%%%%%%%%%%%%%%%%%%%%%%%%%%%%%%%%%%%%%%%%%%%%%%%%%%%%%%%%%%%%

%%%%%%%%%%%%%%%%%%%%%%%%%%%%%%%%%%%%%%%%%%%%%%%%%%%%%%%%%%%%%%%%%%%%%%%%%%%%%%%%%%%%%%%%%%%%%%%%%%
%%%%%%%%%%%%%%%%%%%%%%%%%%%%%%%%%%%%%%%%%%%%%
%%%%%%%%%%%%%%%%%%%%%%%%%%%%%%%%%%%%%%%%%%%%%%%%%%%%%

\section{Product of curves}
This section considers the product of two very general curves $C\times D$ of genus $g\geq 3$ and the dominant rational map
$f:C\times D\dashrightarrow S$ where $S$ is a smooth projective surface with $0\le {\rm kod}(S)\le 1.$

 %%%%%%%%%%%%%%%%%%%%%%%%%%%%%%%%%%%%%%%%%%%%%%%%%%%%%%%%%%%%%%%%%%%%%%%%%%%%%%%%%%%%%%%%%%%%%%%%%%%%%%%%%%%%%%%%%%%%%%%%%%%%%%%%%%%%%%%%%%%%%%%%%%%%%%%%%%%%%%%

\subsection{Hodge structure on abelian \'etale covering of curves} \label{2.1}
In this subsection we let $C$ and $D$  be curves of genus $g_C\geq 3$ and $g_D\geq 3$ respectively. We provide a proof of the proposition~\ref{hs20}, which is the main technical result of this paper. We will deal here with the infinitesimal variation of the Hodge structure of abelian \'etale covering of $C'\times D'\to C\times D$. We will show the existence of decomposable elements in any irreducible sub-Hodge structure of $H^2(C'\times D')$ orthogonal to the classes of the algebraic cycles. These decomposable elements provides the role of the {\em highest weights} and allows to understand these sub-hodge structures.

\begin{definition}
We call a curve $C$  pairwise torsion free (PTF) if
for any pair of distinct points $A, B\in C$ with $A\neq B,$
$nA\not\equiv nB$ for any nonzero integers $n$ i.e. $\cO_C(A-B)$ is not a torsion element
of ${\rm Pic}^0(C).$
\end{definition}

\begin{remark} A PTF curve $C$ is not hyperelliptic, and its genus $g_C\geq 3.$ From the computation of the moduli of the Hurwitz schemes of a curve having a map to $\bP^1$ with $2$ total ramification points, we see that the moduli of curves which are not PTF depends on $2g-1$ parameters. Then a very general curve of genus $\geq 3$  is a PTF curve.
\end{remark}
We now prove two simple algebraic lemmas that will be helpful in studying the infinitesimal variation of Hodge structure.

\begin{lemma} \label{ind}
Let $C$ be a PTF curve of genus $g\geq 3$, and let $L$ and $M$ be torsion line bundles in ${\rm Pic}^0(C).$
Let $\alpha\in H^0(C,\omega_C(L))$ and $\beta\in H^0(C,\omega_C(M))$ be non trivial
sections.  Assume that either $L\neq M$ or $L=M$ with $\alpha\neq \lambda \beta$ for any $\lambda\in \bC.$ Then
there exists $\eta\in H^1(C,T_C)$ such that the cup-product $\eta \cdot \alpha\neq 0\in H^1(C,L)$ and $\eta\cdot\beta=0\in H^1(C,M).$
\end{lemma}
\begin{proof}
We consider the following subspaces $H_\beta$ and $H_\alpha$ of $H^0(C,\omega_C^2):$
 \[H_\beta=\beta \!\cdot\! H^0(C,\omega_C(-M)) ;\  \ \ H_\alpha=\alpha\! \cdot \!H^0(C,\omega_C(-L)).\]  If we can prove that $H_\alpha$ is not contained in $H_\beta,$ then by Serre duality we can find an element $\eta\in H^1(T_C)\cong H^0(C,\omega_C^2)^{\vee} $ such that for $\eta: H^0(C,\omega_C^2) \to \bC$  it holds
$\ker(\eta)\not\supset H_\alpha,$ but  $\ker(\eta)\supset H_\beta .$

Let $E$ and $F$ be the divisor of $\alpha$ and $\beta$ respectively. Let $G$ be an effective divisor of the maximal degree(with multiplicity)
contained in $E$ and $F.$ Since $\alpha$ and $\beta$ are not proportional
$\deg G<2g-2$. Moreover, since $C$ is a PTF curve $\deg G<2g-3.$ In fact if $\deg G=2g-3$ then there exist two distinct points
$P$ and $Q$ such that $G+P\equiv \omega_C(L)$, $G+Q\equiv \omega_C(M)$, and
$L-M\equiv P-Q.$ Since $L$ and $M$ are torsion, $L-M$ is also a torsion line bundle and we get a contradiction.
Next we consider the following exact sequence:
\begin{equation}
0\to G-L-M\to \omega_C(-L)\oplus \omega_C(-M)\stackrel{(\alpha,\beta)}\longrightarrow \omega_C^2(-G)\to 0
\end{equation}

It follows that $H_\alpha+H_\beta$ is the image of the cohomology map
 \[H^0(C, \omega_C(-L))\oplus H^0(C, \omega_C(-M))\stackrel{(\alpha,\beta)}\longrightarrow H^0(C,\omega_C^2(-G))\subset H^0(C,\omega_C^2)\]
and we can identify
 $H_\beta\cap H_\alpha$ and $H^0(C, G-L-M).$
Since $\deg  G-L-M=\deg G \leq 2g-4,$  and $C$ is not hyperelliptic  we obtain
\[\dim H_\alpha \cap H_\beta\leq g-2<\min(\dim H_\alpha,\dim H_\beta),\]
which proves our lemma.
\end{proof}

\begin {lemma} \label{sur}
Let $C$ be a curve of genus $\geq 3.$ Let  $L\in {\rm Pic}^0(C)$, $\alpha\in H^0(C,\omega_C(L)),$
$\alpha\neq 0$. Then
\begin{enumerate}
\item for any $\gamma\in H^1(C,L),$ there is an element $\eta\in H^1(C,T_C)$ such that
$\eta\cdot\alpha=\gamma;$
\item there is an element $\eta\in H^1(T_C)$ such that the map $H^0(C,\omega_C(L))\stackrel{\eta}\longrightarrow  H^1(C,L)$ is an isomorphism.
\end{enumerate}
\end{lemma}
\begin{proof}
Let $E$ be again the divisor of $\alpha $ and consider the exact sequence:
\[0\to T_C \stackrel{\alpha}\longrightarrow L\to L_E\to 0 \quad \text{where $L_E=\cO_E(L)$}.\]
Since $h^1(L_E)=0$ it  follows that   $H^1(T_C)\stackrel{\alpha}\longrightarrow H^1(C,L)$ is surjective.
This proves the first part of the lemma.
For the second part we consider the coboundary map:
\[\partial: H^0(C,L_E)\to H^1(C,T_C).\]
We have that $h^0(L_E)=\deg E=2g-2.$ Therefore $\Gamma_\alpha=\partial(H^0(C,L_E))$ has dimension $2g-3$ if $L=\cO_C$ or $ 2g-2$ if $L\neq \cO_C$. We remark that \[\eta\in \Gamma_\alpha \iff
\alpha\in \ker{\eta}.\] So we have to show that $\bigcup_\alpha \Gamma_\alpha  \neq H^1(C,T_C).$
In fact every element in $H^1(T_C)\setminus \bigcup_\alpha \Gamma_\alpha $ defines an isomorphism
$H^0(C,\omega_C(L))\stackrel{\eta}\longrightarrow  H^1(C,L).$

This will be done by a dimension count passing to the associated projective spaces. Consider the projective space $\bP=\bP H^1(T_C),$ $\dim \bP=3g-4,$ we let $\bP_\alpha \subset \bP$ be the sub-projective space associated to $\Gamma_\alpha.$ We have to show that $\bigcup_\alpha \bP_\alpha\neq \bP.$
Let $\bP_L$ be the projective space associated to  $H^0(C,\omega_C(L)).$
Since  $h^0(C,\omega_C(L))=g-1$ if $L$ is not trivial and $g$ if it is trivial, we have that $\dim\bP_L=g-2$ and $g-1$ respectively.
Now consider the incidence correspondence $\cI \subset \bP\times \bP_L:$
\[\cI =\{((\eta),(\alpha))\in \bP\times \bP_L: \eta\cdot\alpha=0\}. \]
Let $\pi_i$ for $i=1, 2$ be the projections. Since $\pi_2$ is surjective and the fibers are the $\bP_\alpha$ we get $\dim \cI=3g-5.$
Set $Y_{L}=\pi_1(\cI)\subset \bP.$ We notice that
\[Y_L=\bigcup_{(\alpha) \in \bP_L}
\bP_{\alpha}=\{ (\zeta) : \zeta\in \Gamma_\alpha,\ (\alpha)\in \bP_L\}.\]
Therefore we obtain $\dim Y_L\leq 3g-5$  in all cases (and the equality must hold since the $Y_L$ can be defined by the vanishing of a determinant).  This proves the lemma.
\end{proof}

Now we set our notation. We let  $p>1$ be a prime number, $L_1$ be a line bundle on $C$ and $L_2$ be a line bundle on $D,$
such that $L_1^p=\cO_C$ and $L_2^p=\cO_D.$ We assume that $L_1$ and $L_2$ are not trivial, but
the case $L_1=\cO_C$ could be considered and it is simpler. Let $f_1:C'\to C$ and $f_2:D'\to D$ be the \'etale covering associated to $L_1$ and $L_2.$ One has:

\begin{equation} \label{eigen}
f_{1\ast} \cO_{C'}= \sum_{i=0}^{p-1} L_1^i\ \ \ f_{2\ast} \cO_{D'}= \sum_{i=0}^{p-1} L_2^i.\end{equation}

Set $X=C'\times D'.$ We have from the Kunneth decomposition $H^{2,0}(X)\cong H^{1,0}(C')\otimes H^{1,0}(D'),$ then using the above decomposition (\ref{eigen}) :

\[H^{2,0}(X)\cong H^{1,0}(C')\otimes H^{1,0}(D'); \ \ H^{0,2}(X)\cong H^{0,1}(C')\otimes H^{0,1}(D')\]
\[ H^{1.0}(C')\cong \oplus_{i=0}^{p-1} H^0(C,\omega_C(L_1^i));\ \ H^{0,1}(C')\cong \oplus_{i=0}^{p-1} H^1(C,L_1^{-i}) \]
\[ H^{1.0}(D')\cong\oplus_{i=0}^{p-1} H^0(D,\omega_D(L_2^i));\ \ H^{0,1}(D') \cong\oplus_{i=0}^{p-1} H^1(D,L_2^{-i}) .\]
We set
 \[V_{i,j}\cong H^0(C,\omega_C(L_1^i))\otimes H^0(D,\omega_D(L_2^j)),\]
 \[V'_{i,j}\cong H^1(C,L_1^i))\otimes H^1(D,L_2^j).\]

 For the sake of notation we will use the above isomorphism as an identification. Then we will often omit the pull-backs, for instance we  write  $H^0(C,\omega_C(L_1^i))$ instead of  $f_1^\ast H^0(C,\omega_C(L_1^i)),$ etc.
\begin{equation}
H^{2,0}(X)=\oplus_{ij=0}^{p-1}V_{i,j};\ \ \ H^{0,2}(X) =\oplus_{ij=0}^{p-1}V'_{i,j}. \label{dec}
\end{equation}
 One has $V_{i,j}=V_{i+p,j}=V_{i,j+p} $ and by the complex conjugation \[\overline{V_{i,j}}=V'_{-i,-j}=V'_{p-i,p-j}.\]
Then we get the following proposition.

\begin{proposition} \label{hs20}
Let $C$ and $D$ be very general curves of genus  $\geq 3.$ Let $\Lambda\subset H^2(C'\times D')$ be a Hodge substructure.
If $\Lambda^{2,0}\neq 0$ then for some index ${a,b}$, we have $\Lambda^{2,0}\supset V_{a,b}.$
\end{proposition}
\begin{proof}
We remark that the deformations of $C'\times D'$ that preserve the coverings correspond to the deformations of $C$ and of $D.$
In fact let $\pi_1: \cC \to \Delta $ and $\pi_2 : \cD \to \Delta'$ be Kuranishi families of
 $C$ and $D$ respectively, then $\pi: \cC \times \cD \to \Delta\times \Delta'$ is a Kuranishi family of $C\times D.$
 We have $\Delta\times \Delta'$ contractible and a point $p\in  \Delta\times \Delta'$
such that $\pi^{-1}(p)=C\times D.$ The tangent space $T$ of $\Delta\times\Delta'$ at $p$ is identify $H^1(C,T_C)\oplus H^1(D,T_D)$ via the Kodaira-Spencer map.
Since  $\Delta\times \Delta'$ is contractible, the coverings $f_1$ and $f_2$
define a family $\pi': \cC'\times \cD' \to \Delta\times \Delta'$  and an \'etale covering $\cC'\times \cD'\to \cC\times \cD.$ This gives the Kuranishi family of the covering map $f_1\times f_2.$  In particular this induces an injective map
 $ T=H^1(C,T_C)\oplus H^1(D,T_D)\to H^1(C',T_{C'})\oplus H^1(D',T_{D'})$ that gives an identification of $T$ with the infinitesimal deformations of $C'\times D'$  preserving the covering.
\medskip

We shall use a basic result from the infinitesimal variation of  Hodge structures (cf. \cite{V}). Let $\cH =R^2\pi_\ast \bC$ be the variation of Hodge structures defined by the Kuranishi family on the second complex cohomology. Let ${\co}$ be the Gauss Manin  connection defined on the sections of  $\cH.$  Let $\mathcal L \subset \cH$ be a Hodge substructure.
 It follows that $\co$ must preserve the sections of $\mathcal L.$
Consider the Hodge decomposition
$$H^2(C'\times D',\bC)=H^{2,0}\oplus H^{1,1}\oplus H^{0,2}.$$
The infinitesimal variation
 $\co^{1.1}: H^{2,0}\times T\to H^{1,1}$ and $\co^{0.2}:H^{1,1}\times T\to H^{0,2}$
are defined via the Kodaira-Spencer map.
It means  that if $s$ is a section of $\cH^{p,q}\subset \cH $ and $\phi\in T$
then $$(\co_\phi s)^{p-1.q+1}(p)=\phi\cdot s(p),$$
where $\cdot$ is the cup-product.
This follows from the Griffiths transversality theorem (cf. Section 2.2 in \cite{V}).
Let $\Lambda=\mathcal L_p\subseteq H^2(C'\times D',\bC)$ be the Hodge substructure. Then we have
$\Lambda^{p,q}\subset H^{p,q},$ and
then $\co^{1.1}: \Lambda^{2,0}\times T\to \Lambda^{1,1}$ and  $\co^{0.2}:\Lambda^{1,1} \times T\to \Lambda^{0,2}.$

Therefore for any pair of the infinitesimal deformations, $\eta$ and $\zeta\in T,$  we get:\begin{equation}\eta \cdot \zeta \cdot \Lambda^{2,0}\subset \Lambda^{0,2}\label{cond}.\end{equation}
We will consider $\eta\in H^1(T_C)\subset T$
and $\zeta\in H^1(T_D)\subset T.$
\medskip

Our first aim is to show that if the space $\Lambda^{2,0}\neq 0$ then it contains some non trivial decomposable element.
For any $\Gamma \in \Lambda^{2,0}\subset H^{2,0}(X),$  we write $\Gamma=\sum\gamma_{i,j},$
$\gamma_{i,j}\in V_{i,j}.$
First we will show that there is an element $\Gamma\in  \Lambda^{2,0} $ such that the components $ \gamma_{i,j}\in V_{i,j}$ are all decomposable:
$\gamma_{i,j}=\alpha_{ij}\wedge \beta_{ij}.$ (Since our element are forms we use the wedge product instead of $\otimes$).
Starting with any $\Omega=\sum\omega_{i,j} \neq 0$ and take an index $i,j$ such that $\omega_{i,j}\neq 0$. Write
\[\omega_{i,j}=\sum_{k=1}^s \alpha_k\wedge\beta_k.\]
We assume that the $\alpha_k$ and the $\beta_k$ are independent,
that is the rank of tensor $\omega_{i,j}$ to be $s.$ If $s=1$ there is noting to do for the index $i.j$. Assume $s>1.$
In particular $\alpha_1$ and $\alpha_2$ are not proportional.
Now by Lemma~\ref{ind} we find $\eta\in H^1(C,T_C)$  such that $\eta\cdot{\alpha_1}=0$ and $\eta\cdot{\alpha_2}\neq 0 $. And by Lemma~\ref{sur} we have $\zeta\in H^1(D,T_D)$ such that
$\zeta:H^0(D,\omega_D(L_2^j))\to H^1(D, L_2^j)$ is an isomorphism.
Since the Hodge structure $\Lambda$  must deform with $C\times D,$ $\Theta=\zeta\cdot (\eta\cdot  \Omega)=\eta\cdot (\zeta\cdot  \Omega)\in \Lambda^{0,2}.$
The infinitesimal variation of Hodge structure is given by the cup-product, and the
cup product commutes with the decomposition (\ref{dec}) we get
that $ \Theta_{i,j}=\zeta\cdot (\eta\cdot  \omega_{i,j}).$
Then
\begin{equation}\label{IVH2}
\Theta_{i,j}=\sum^s_{k=1}\eta\cdot \alpha_k\wedge\zeta\cdot \beta_k.\end{equation}
 We remark that the rank cannot increase, moreover since $\eta\cdot \alpha_1=0$, $\Theta$ has rank $\leq s-1:$
\[\Theta_{i,j}=\sum^s_{k=2}\eta\cdot \alpha_k\wedge\zeta\cdot \beta_k.\]
 Since $\zeta$ is an isomorphism the vectors $\zeta\cdot \beta_k$ are all independent and
$\eta\cdot \alpha_2\neq 0$, and then  we obtain that $\Theta_{i,j}\neq 0.$ Now we use the complex conjugation.
We define $\widetilde{\Omega}= \overline {\Theta}\in \Lambda^{2,0}$, moreover $\widetilde{\Omega'}_{p-i,p-j}:$
\[\widetilde{\Omega}_{p-i,p-j}=\sum^s_{k=2}\overline{(\eta\cdot \alpha_k)}\wedge\overline{(\zeta\cdot \beta_k)}.\]
is non trivial
of rank $\leq s-1.$ As the formula (\ref{IVH2}) shows the rank of the component of our tensor cannot increase under the cup product action. We can repeat the above operation for any index $ i',j'.$ Finally we find
$0\neq\Omega'\in \Lambda^{2,0}$ such that
$\Omega'=\sum \alpha_{ij}\wedge \beta_{ij}.$
\medskip

Next we would like to show that we can find a decomposable non trivial element
$\alpha_{ij}\wedge\beta_{ij} \in \Lambda^{2,0}.$
We start this time from $\Omega'=\sum_{ij}\alpha_{ij}\wedge \beta_{ij}.$ If its rank is one then we have done. We may assume the rank is $r\geq 2,$ and then that for two pairs of different indices $ij$ and $i'j'$
$\alpha_{ij}\wedge \beta_{ij}\neq 0$ and $\alpha_{i'j'}\wedge \beta_{i'j'}\neq 0.$ We have
$i\neq i'$  or  $j\neq j'.$ By the symmetry of the hypothesis on $C$ and $D$ we may assume $i\neq i'.$ By Lemma~\ref{ind} we can find
$\eta\in H^1(C,T_C)$ such that $\eta\cdot\alpha_{ij}\neq 0$ and $\eta \cdot \alpha_{i'j'}=0,$ and $\zeta\in H^1(D,T_D)$ such that $\zeta\cdot \beta_{ij}\neq 0.$ It follows then $\Theta=\zeta\cdot (\eta\cdot  \Omega)\in \Lambda^{0,2}:$
\[\Theta'_{i,j}=\zeta\cdot (\eta\cdot  \Omega')_{i,j}=\eta\cdot \alpha_{ij}\wedge\zeta\cdot \beta_{ij}\neq 0\]
and $\Theta'_{i',j'}=0.$  Taking $\Omega''=\overline \Theta'$ we get a non trivial element in $\Lambda^{2,0}$ with rank $0<r'<r$. Repeating the operation we can find an element
\[0\neq \alpha\wedge\beta\in V_{a,b}\cap \Lambda^{2,0}.\] \label{IVHS}

Finally we see that $\eta\cdot\alpha\wedge \zeta \cdot \beta $ belongs to $\Lambda^{0,2}$ for all $\eta\in H^1(C,T_C)$ and
$\zeta \in H^1(D,T_D).$ Using the first part of Lemma~\ref{sur}  we get that all the decomposable elements in $V'_{a,b}$ are in $\Lambda^{0,2}.$ It concludes
$V'_{a,b}\subset \Lambda^{0,2}.$ Therefore $V_{p-a,p-b}\subset \Lambda^{2,0}.$
\end{proof}

Now we can prove the following.
\begin{theorem} \label{kodsh}
Let $C$ and $D$ are very general curves of genus $g\geq 3$. Let $f_1:C'\to C$ and $f_2:D'\to D$ as above where $f_i$ for $i=1, 2$ are \'etale coverings of prime order $p$. Let $X=C'\times D'.$ Assume that $\Lambda\subset H^2(X) $ is a Hodge substructure such that $\Lambda^{2,0} \neq 0$ ($\Lambda^{2,0}\subset H^{2,0}=H^0(X,K_X)$). Let $|\Lambda^{2,0}|$ be the corresponding sublinear series of $H^0(X, K_X)$.
Then the image of the rational map $X=C'\times D'\dashrightarrow |\Lambda^{2,0}|$  has dimension $2.$
\end{theorem}
\begin{proof}
From Proposition~\ref{hs20} we can find indices $i,j$ such that $\Lambda^{2,0}\supset V_{i,j} =H^0(C,\omega_C(L_1^i)))\otimes H^0(D,\omega_D(L_2^j))$
It is enough to show that the image of \[X\dashrightarrow  |V_{i.j}|\] has dimensions $2.$ This is clear since it factorizes through
\[ C\times D\dashrightarrow |H^0(C,\omega_C(L_1^i))\otimes  H^0(D,\omega_D(L_2^j))|.\]
Since $C$ and $D$ have both genus $\geq 3$, $h^0(C,\omega_C(L_1^i))\geq 2$ and  $h^0(D,\omega_D(L_2^j))\geq 2.$
The result follows by using the Segre embedding.
\end{proof}

For the sake of completeness, we can consider $P(C',C)\subset H^1(C)$, and $P(D',D)\subset H^1(D)$ respectively. These are the Hodge structures corresponding to the kernel of the norm mappings $H^1(C')\to H^1(C)$, and $H^1(D')\to H^1(D))$ respectively.
Clearly $P(C',C)$  corresponds to the Prym variety of the covering $C'\to C.$
The result of Theorem~\ref{IVHS} and the examination of the monodromy on the torsion line bundle imply the following proposition.

\begin{proposition}
Let $C,D,C'D'$ be as above. Assume that the covering $C'\to C$ and $D'\to D$ are non-trivial. There are exactly $4$ irreducible Hodge substructures of $H^2(C'\times D')$  with non trivial $(2,0)$ part:
$H^1(C)\otimes H^1(D),$ $P(C',C)\otimes H^1(D),$ $H^1(C)\otimes P(D',D)$ and $P(C',C)\otimes P(D',D).$
\end{proposition}

We remark that
\[(P(C',C)\otimes P(D',D))^{2,0}= \oplus V_{i,j},\ \ 0<i,j<p.\]

%%%%%%%%%%%%%%%%%%%%%%%%%%%%%%%%%%%%%%%%%%%%%%%%%%%%%%%%%%%%%%%%%%%%%%%%%%%%%%%%%%%%%%%%%%%%%%%%%%%%%%%%%%%%%%%%%%%%%%%%%%%%
\subsection{Torsion freeness}

This subsection proves our main technical result, which is Theorem~\ref{tor}. Let $C$ and $D$ be two very general curves of genus $g\geq 3$.  Suppose there is a dominant rational map
$f:C\times D\dashrightarrow S$ where $S$ is a smooth projective surface with $0\le {\rm kod}(S)\le 1.$

\begin{lemma} \label{pgs}
We have $p_g(S)=q(S)=0$ and $S$ has an elliptic fibration $\pi: S\to P^1.$
\end{lemma}
\begin{proof}
We first prove that the geometric genus vanishes.
The lemma follows from by the irreducibility of the transcendental Hodge structure of $H^2$ of a very general product. That is for a very general product, the N\'eron Severi group $NS(C\times D)$ of $C\times D$ has rank two generated by the classes of $C\times \{x\}$ and $\{y\}\times D$ for $(y,x)\in C\times D.$ The Hodge structure orthogonal to the image of the algebraic cycle is irreducible (see for instance Lemma 2.7 and Lemma 4.1 in \cite{BaPi} (or apply the argument of the previous subsection to $L_1=\cO_C$ and $L_2=\cO_D$). Since the map
$f^\ast :H^{2,0}(S)\to H^{2,0}(C\times D)$ is injective, then either $H^{2,0}(S)=0$ or $f^\ast$ is an isomorphism. The second case is impossible, otherwise the canonical map would factorize through $g$ and the canonical map of $S.$ Now the canonical map of $C\times D$ gives embedding, but the canonical image of $S$ is not a surface since Kodaira dimension of $S$ is $<2.$
Now we prove that $q(S)=0.$ Assume by contradiction that $q(S)>0.$ Let $a: S\to alb(S)$ be the albanese map. Since $p_g(S)=0$, the image of $a$ must be a smooth curve $E$ of genus $q(S)>0.$ The composition of $f$ with $a$ define a map $C\times D\to alb(S)$ and therefore $f_1:C\to E$ and $f_2:D\to E.$ Since $C$ and $D$ are very general it follows that one of  these is an isomorphism and the other one is constant. Suppose that $E\cong C$ and $f_2$ is constant, we have $q(S)=g_C>1.$ It follows that Kodaira dimension $S$ is $1$ and that $a:S\to E$ is an elliptic fibration. Set $E=C.$
Since $f_2$ is constant, we have a commutative fibration diagram
\[\xymatrix@C-1pc{
  C\times D \ar@{.>}[rr]^{f} \ar[dr]_{g} && S \ar[dl]^{a} \\
  & C.
}\]
It follows that the curve $D$ maps onto to elliptic fibration $a.$ Since $D$
is very general this is impossible, and we obtain a contradiction.

The final statement follows from the theory of classification of the surfaces (see for instance \cite{bpv}).
\end{proof}

Since the fundamental group $C\times D$ is not abelian, the fundamental group of $S$ need not to be abelian. But we will prove the following by using Theorem~\ref{kodsh}.

\begin{theorem} \label{tor}
We have $H_1(S,\bZ)=0,$ that is  $H^2(S,\bZ)$ is torsion free and $q(S)=0.$
\end{theorem}
\begin{proof} Let $S$ be a minimal surface of Kodaira dimension $0$ or $1$, and $K_S$ be the canonical divisor of $S$. Then we have $K_S^2=0$. We assume by contradiction that there is a dominant rational map
\[f: C\times D \dashrightarrow S\]
where $C$ and $D$ are very general curves and $H_1(S,\bZ)\neq 0.$ Since $H_1(S,\bZ)$ is a finite generated abelian group we can find a prime $p$, and a surjection
$H_1(S,\bZ)\to \bZ/p\bZ$. Therefore we have a surjection $\psi: \pi_1(S)\to  \bZ/p\bZ.$
Let $q:S'\to S$ be the \'etale covering associated to  $\ker(\psi).$
Let  $K_{S'}$ be the canonical class of  $S'$ we have
$K^2_{S'}= f^\ast(K_S)^2=0,$ then it follows that $S'$ has Kodaira dimension $0$ or $1.$
We also have the following proportionality  result for the topological and the holomorphic Euler characteristic:\ $c_2(S')=p\cdot c_2(S)$ and $\chi(\cO_{S'})=p\cdot \chi(\cO_S).$
\medskip
By composition with
$f$ (after a suitable resolution) we get \[\psi': \pi_1(C\times D)=\pi_1(C)\times \pi_1(D)\to \pi_1(S)\to \bZ/p\bZ.\]

{\bf Claim}: The map $\psi'$ is surjective.

First note that either $\psi'$ is a surjection or it is the zero map. Assume by contradiction $\psi'=0.$  In this case we have a lifting map $f':C\times D\dashrightarrow S'.$ Then by Lemma $\ref{pgs}$ either $f'$ is  birational or $p_g(S')=q(S')=0.$ The first case is not possible since $S'$ has Kodaira dimension $\leq 1$.
\smallskip

Now assume  $p_g(S')=q(S')=0.$ We get $\chi(\cO_{S'})=1-q(S')-p_g(S')=1=p\cdot \chi(\cO_S)=p.$  This also  gives a contradiction.

\medskip
We have proven $\psi': \pi_1(C\times D)\to \bZ/p\bZ$ is surjective.
This gives two maps $\pi_1(C)\to \bZ/p\bZ$ and $\pi_1(D)\to \bZ/p\bZ.$
They produce two \'etale coverings $C'\to C$ and $D'\to D$ (one of the two coverings can be trivial).

We have remarked that ${\rm kod}(S')\leq 1$  since $K_{S'}^2=0,$ then we have $\chi(\cO_{S})=p_g(S)-q(S)+1=1$ by Lemma~\ref{pgs} and therefore  $\chi(\cO_{S'})=p\chi(\cO_{S})=1.$ It follows then
$\dim H^{2,0}(S')=p_g(S')=\chi(\cO_{S'})+q(S')-1\geq p-1\geq 1.$

Now we take a resolution $f''$ of $f'$, that is a smooth projective surface $Y,$ a birational map $\epsilon: Y\to C'\times D'$ and a morphism
$f'': Y\to S'$ such that $ f''\circ \epsilon=f'$. We define then the substructure
$\Lambda\subset H^2(C'\times D') $ as follows :
 \[\Lambda=\epsilon_\ast f''^\ast(H^2(S'))\]
 is a Hodge substructure of $H^2(C'\times D')$
with non zero $(2,0)$ part,
 $\Lambda^{2,0}\neq 0.$ In fact $\epsilon_\ast: H^{2,0}(Y)\to H^{2,0}(C'\times D')$ is injective.
  By Theorem \ref{kodsh} we get that the image
$\kappa: C'\times D'\dashrightarrow |\Lambda^{2,0}|$ has dimension $2.$ On the other hand
$\Lambda^{2,0}= \epsilon_\ast f''^\ast(H^{2,0}(S'))= f'^\ast(H^{2,0}(S')),$ then $ \kappa$ must factorize:
 $\kappa= f'\circ \gamma$ where $\gamma$ is the canonical map of $S'.$ Since
Kodaira dimension of $S'$ is $\leq 1$ we get a contradiction.
\end{proof}

%%%%%%%%%%%%%%%%%%%%%%%%%%%%%%%%%%%%%%%%%%%%%%%%%%%%%%%%%%%%%%%%%%%%%%%%%%%%%%%%%%%%%%%%%%%%%%%%%%%%%%%%%%%%%%%%%%%%%%%%%%%%%%%%%%%%%%%%%%%%%%%%%%%%%%%%%%%%%%%%%%%%%%%%%%%%%%%%%%%%%%%%%%%%%%%%%%%%%%%%%%%%%%%%%%%%%%%%%%%%%%%%%%%%%%%%%%%%%%%%%%%%%%%%%%%%%%%%%%%%%%%%
\section{Dominant rational map}

Combining Theorem~\ref{tor} with the result in \cite{BaPi} we will proof our main theorem. We also recall the fact that there is no map from a very general curve $C$ of genus $g(C)\geq 7$
to $S$ which is birational on its image if $S$ is a elliptic surface with $p_g(S)=q(S)=0$ of Kodaira dimension 1 and if ${\rm Pic}(S)$ is torsion free (Remark 4.4. in \cite{LeeP}).
\begin{theorem}
Let $C$ and $D$ be very general curves of genus $g_C\geq 7$ and $g_D\geq 3$ respectively. Then
there is no dominant rational map of degree $>1$ from $C\times D$ to $S$ where $S$ is a smooth projective surface of ${\rm kod}(S) \geq 0.$
\end{theorem}

\begin{proof}
Let  $S$ be a smooth projective surface of ${\rm kod}(S) \geq 0$ and let $F:C\times D\dashrightarrow  S$ be a dominant map. We have to show that ${\rm kod}(S)= -\infty .$
In  \cite {BaPi}  it is shown that there is no dominant rational map from $C\times D\dashrightarrow  S$ where $C$ and $D$ are generic and $S$ is of general type. By Theorem~\ref{tor} we may assume that $p_g(S)=0$ and $H_1(S,\bZ)=0,$ this takes care of the case of surfaces with Kodaira dimension 0, because no such surface exists. So we have only to consider the case where ${\rm kod}(S)= 1$, $p_g=q=0$, and  ${\rm Pic}(S)$ is torsion free in force of Theorem~\ref{tor}. Remark 4.4 in \cite{LeeP} shows that there is no map from a very general curve $C$ of genus $g(C)\geq 7$
to $S$ which is birational on its image. This implies the map $F$ restricted to the fiber $C\times{t}$ for a general $t\in D$ cannot be birational. Therefore $F(C\times{t})$ is a rational curve since
$C$ is very general. Therefore $S$ is a ruled surface because a nonempty Zariski open set of $S$ is covered by rational curves, this gives a contradiction.
 \end{proof}
%%%%%%%%%%%%%%%%%%%%%%%%%%%%%%%%%%%%%%%%%%%%%%%%%%%%%%%%%%%%%%%%%%%%%%%%%%%%%%%%%%%%%%%%%%%%%%%%%%%%%%%%%%%%%%%%%%%%%%%%%%%%%%%%%%%%%%%%%%%%%%%%%%%%%%%%%%%%%%%%%%%%%%%%%%%%%%%%%%%%%%%%%%%%%%%%%%%%%%%%%%%%%%%%%%%%%%%%%%%%%%%%%%%%%%%%%%%%%%%%%%%%%%%%%%%%%%%%%%%%%%%%%%%%%%%%%%%%%%%%%%%%%%%%%%%%%%%%%%%%%%%%%%%%%%%%%%%%%%%%%%%%%%%%%%
\section{Curves on an elliptic surface and symmetric products}
%%%%%%%%%%%%%%%%%%%%%%%%%%%%%%%%%%%%%%%%%%%%%%%%%%%%%%%%%%%%%%%%%%%%%%%%%%%%%%%%%%%%%%%%%%%%%%%%%%%%%%%%%%%%%%%%%%%%%%%%%%%%%%%%%%%%

\subsection{Curves on an elliptic surface}  In this subsection we will slightly improve the result of \cite{LeeP} on deformation on curves on elliptic surface. We will give an application of  this result in the next subsection.

Let $\pi:S\to \bP^1$ be an elliptic surface (relatively minimal) with $p_g=q=0$ and of Kodaira dimension $1.$
Let $C$ be a smooth projective curve of genus $g>1$. Let
 \[\kappa: C\to S \] be a birational immersion, that is the map $C\to \kappa(C)$ is birational.
 Let $U$ be an irreducible component of the Kuranishi family of $\kappa.$
 \begin{proposition}\label{dimension}
Assume that  $C$ is neither a hyperelliptic nor a trigonal curve of genus $g>2.$ Then $\dim U\leq g-2.$
\end{proposition}
\begin{proof}
We assume that $C$ is a general curve in the family. Since $K_S$ is nef, non-trivial and semi-ample, we have that $\deg \kappa^\ast(K_S)\geq 1.$
The differential of $\kappa$  induces an exact sequence
\[0\to T_C\to \kappa^\ast T_S\to N\to 0.\]
Let  $N_{tors}$ be the torsion of the normal bundle $N$ and  $N'=N/N_{tors}$ be the quotient.
We have (\cite{ac}, Chapter XXI in \cite{ACG})
$\dim(U)\leq h^0(N'),$ $N'\subset K_C\otimes \kappa^\ast K_S^{-1}.$
Since the curve $C$ is not hyperelliptic, we have to consider only the case where  $\deg \kappa^\ast(K_S)=1,$ by Proposition 2.2 in \cite{LeeP}.
By examination of the multiple fibers we see that  $K_S$ is numerically a multiple of line bundle $\lambda$,   $K_S=\rho\lambda$ where $\rho \geq 1$. Then $\deg \kappa^\ast(K_S)=1$ implies $\rho=1,$  this is possible only
in the following two cases, both with only two multiple fibers (cf. Remark 4.2 in \cite{LeeP}):
\begin{enumerate}
\item the case $(2,3)$,
\item the case $(2,4)$.
\end{enumerate}

In the $(2,3)$ case $S=S_{2,3}$ has multiple fibers $F_1$ and $F_2$ in the elliptic fibration where $2F_1=F=3F_2$ for a general fiber $F$.  Then
$F_1=3K_S$ and $F_2=2K_S$, and it implies
 $K_S=F_1-F_2.$ If  $h^0(\kappa^\ast(K_S))=1$ then $\kappa^\ast(K_S)=\cO_C(P)$ for $P\in C.$ On the other hand since $\kappa(C)$ is not contained in the fiber $\kappa^\ast(F_1)=\cO_C(G+E+F)$
and $\kappa^\ast(F_2)=\cO_C(R+S)$ where $G,E,F,R, S$ are point of $C.$
Therefore $R+S+P\equiv G+E+F,$ but since $C$ is not trigonal we would have then $\{ R,S,P\}=\{ G,E,F\},$ which is impossible since the two fibers are distinct.\\
 In the second case
we have $K_S=F-F_1-F_2=F_1-F_2=3F_2-F_1,$ $2F_1=F=4F_2.$
If $\kappa^\ast(K_S)=\cO_C(P)$ we have $\kappa^\ast(F_1)=\cO_C(A+B)$
and $\kappa^\ast(F_2)=\cO_C(Q)$. Then $A+B\equiv P+Q$, which is also impossible.
\end{proof}

\begin{example} \label{ex}
We note that the above result is sharp.
Let us consider the case $(2.3)$, that is $S=S_{2,3}$. The elliptic surface  $S=S_{2,3}$ is simply connected with $p_g=q=0$ \cite{Dol}. We have $6K_S=F$ and
$3K_S=F_1$ and $2K_S=F_2$. By Poincar\`e duality the coset  $H=\{\alpha \in H^2(S,\bZ):
K_S\cdot \alpha=1\}$ is not empty. We notice that
\begin{enumerate}
\item if  $\alpha \in H$ then $ \alpha^2=2s+1$ is odd,
\item if $\alpha \in H$ then $\alpha+K_S\in H$,
\item if $\alpha \in H$ then $(\alpha+ K_S)^2=\alpha^2+2$.
\end{enumerate}
We choose $\alpha \in H$ such that $\alpha^2=-3$ and consider the line bundle
$L_r=\alpha\otimes rK_S$. Then $L_r^2=2r-3$, and
we have from the Riemann Roch theorem
$\chi(L_r)= r-1$ and $ L_r(L_r+K_S)=2(r-1)$.

By the duality
$h^2(L_r)=h^0(K_S-L_r)=0$ since $K_S-L_r= (1-r)K_S-\alpha$, $-\alpha\cdot F=-6<0$, and $F$ is semi-ample.
It follows that for $g> 1$ $L_g$ has a global section of arithmetic genus $g$ and the family depends upon
 $h^0(L_g)-1\geq g-2$ parameters.
\end{example}
%%%%%%%%%%%%%%%%%%%%%%%%%%%%%%%%%%%%%%%%%%%%%%%%%%%%%%%%%%%%%%%%%%%%%%%%%%%%%%%%%%%%%%%%%%%%%%%%%%%%%%%%%%%%%%%%%%%%%%%%%%%%%%%%%%%%%%%%%%%%%%%%%%%%%%%%%%%%%%%%%%%%%%%%%%%%%%%%%%%%%%%%%%%%%%%%%%%%%%%%%%%%%%%%%%%%%%%%%%%%%%%%%%%%%%%%%%%%%%%%%%%%%%%%%%%%%%%%%%%%%%%%%%%%%%%%%%%%%%%%%%%%%%%%%%%%%%%%%%%%%%%%%%%%%%%%%%%%%%%%%%%%%%%%%%
\subsection{Symmetric product}
This subsection is devoted to prove the following.
\begin{proposition}\label{sim2}
Let $C$ be a very general curve of genus $g\geq 6.$ Let $C_2$ be the $2$-symmetric product of $C.$
Let $S$ be a surface of Kodaira dimension $1.$ Then there is no dominant rational map  $f: C_2 \dashrightarrow S.$
\end{proposition}
\begin{proof}

Assume that  $f: C_2 \dashrightarrow S$ is a dominant rational map and let $f':Y\to S$ be a regular morphism,
 where $f'$ is  a resolution of $f.$ That is there is a birational morphism $\epsilon: Y\to C_2.$
%We recall that there is an Zariski open set $W$ of $Y$ such that $\epsilon|_W: W\to \epsilon(W)$ is
%biregular and $\epsilon(W)=C_2\setminus Z$ where $\dim Z=0,$ that is a $Z$ is a finite set.
By the irreducibility of the transcendental Hodge structure of $H^2(C_2)$ (see also Subsection~\ref{2.1}) and the fact that the canonical map is birational for $C_2$ we get $p_g(S)=q(S)=0.$ Since ${\rm kod}(S)= 1$ we
may assume that $\pi:S\to \bP^1$ is a relatively minimal elliptic fibration. We claim that the map on the fundamental groups $f'_\ast : \pi_1(Y)\to  \pi_1(S)$ is surjective.

We assume by contradiction $f'_\ast : \pi_1(Y)\to  \pi_1(S)$ is not surjective.  In fact we know that the image $\Gamma$ of $f'_\ast$ has a finite index.
Assume by contradiction that $\Gamma\neq  \pi_1(S).$ Let $\pi: S'\to S$ be the associated covering of degree $m=[\pi_1(S):\Gamma].$
By the construction we can find a lift $f'':Y\to S'$ such that $\pi\circ f''=f'.$ By the proportionality we get
$\chi(\cO_{S'})=m\chi(\cO_{S})=m.$ Therefore $p_g\geq m-1.$ But again we have than $p_g(S')=0$ or $S'\cong Y$ In the first case we obtain $m=1$ and therefore a contradiction, in the second one case we may assume that $C_2$ is birational to $S'$ so we may assume $S'=Y$ and $f''=f'$ the resolution of $f.$ Since the map $\pi$ is \'etale, $K_{S'}$ is nef and $K_{S'}^2=K_S^2=0$.
But since $Y$ is birational to $C_2$ that is a surface of general type, so we find a contradiction.

Since $\pi_1(Y)\cong \pi_1(C_2) \cong H_1(C,\bZ)$, $\pi_1(S)$ is abelian. Therefore $\pi: S\to\bP^1$ must have exactly two multiple fibers since it is non-rational \cite{Dol}. We have $H^2(T_S)=0,$ and the deformation of $S$ depend upon $10$ parameters (see Proposition 3.4 in \cite{LeeP}).
Let  $C$ be a very general curve of genus $g\ge 6$. For any point $P$ of $C$, let $c_P:C\to C_2$ be the embedding (it is called a coordinate curve) by
\[c_P(Q)=P+Q.\]  Set $C_P =c_P(C))$ (it is called a coordinate curve).
We remark that $C_P^2=1$ and that give an ample divisor of $C_P.$

Consider the composition map $f_P=f\circ c_P: C\to S.$
Then for general $P$, $f_P$ is  birational onto its image: otherwise, $C$ can be mapped non-trivially only on
$\bP^1$ since $C$ is a very general curve in the moduli space of curves of genus $g$.
 But $f$ is dominant and $S$ is not ruled.

It follows that we have a family of birational immersions of dimension $3g-3+1$
(the dimension of the moduli  plus the one due to  the coordinates curves ${C_P}_{\{P\in C\}}$ in
$C_2$). We remark that $C$ is not trigonal. Since the deformations of our surfaces depend on $10$ dimensional moduli, in some fixed surface $S$ we must find at least a $3g-2-10$ dimensional family of birational immersions. Therefore by Proposition~\ref{dimension} we have \[3g-2-10\leq g-2\]
$2g\leq 10$ we get $g\leq 5.$
\end{proof}

\begin{remark} The same argument for the surjection of the map on the fundamental groups in the proof of Proposition~\ref{sim2} works in the proof of Theorem~\ref{tor}.
\end{remark}

\begin{remark} When $g=5$ the above example \ref{ex} shows that there is a three dimensional family
in $S_{2.3}.$ Then there is a family of $3g-2 =13$ dimensional deformations into the family of elliptic surfaces $S_{2.3}.$ It could be however that on a general $S_{2.3}$ we have no isotrivial deformation, that is deformation of constant moduli.
\end{remark}

For the maps between $C_2$ and surfaces of general type, we have
the inequality $3g-2-19\leq g-2$ \cite{gp}: that is $g< 10$. The case $g=9,8$ can be also excluded.
When ${\rm kod}(S)=0$ with  $p_g=q=0,$ the case of Enriques surfaces can be also excluded by the dimensional count as above or by using the method in the proof of Theorem~\ref{tor}.

In conclusion we get the following.
\begin{theorem}
Let $C$ be a very general curve of genus $\geq 10$ and let $f: C_2 \dashrightarrow S$ be a dominant rational map of degree $\geq 2$. Then
$S$ is a rational surface.
\end{theorem}

%%%%%%%%%%%%%%%%%%%%%%%%%%%%%%%%%%%%%%%%%%%%%%%%%%%%%%%%%%%%%%%%%%%%%%%%%%%%%%%%%%%%%%%%%%%%%%%%%%%%%%%%%%%%%%%%%%%%%%%%%%%%
%%%%%%%%%%%%%%%%%%%%%%%%%%%%%%%%%%%%%%%%%%
%%%%%%%%%%%%%%%%%%%%%%%%%%%%%%%%%%%%%%%%%%

\end{document}